\documentclass[12pt,a4paper]{article}
\usepackage{graphicx,epsfig, cancel}
\usepackage{amsfonts,natbib}
\usepackage{amsthm, amsmath}
\usepackage{amssymb}
\usepackage{bbm}
\usepackage{enumerate}
\usepackage{xcolor}
\colorlet{red}{black}
\usepackage{textcomp}
\usepackage{listings}
\usepackage{hyperref}
\usepackage{pdfpages}
\newtheorem{theorem}{Theorem}[section]

\newtheorem{lemma}[theorem]{Lemma}
\newtheorem{corollary}[theorem]{Corollary}

\newtheorem{definition}[theorem]{Definition}

\newtheorem{note}[theorem]{Note}
\usepackage{tikz}
\usetikzlibrary{matrix}

\DeclareMathOperator{\osc}{osc}
\newcommand{\R}{{\mathbb R}}
\newcommand{\pvar}{p{\mathrm{-var}}}
\newcommand{\var}{\mathrm{-var}}

\newtheorem{remark}[theorem]{Remark}

\title{An elementary approach to delta and gamma hedging, and the Black--Scholes PDE}
\author{John Armstrong\thanks{Email: john.armstrong@kcl.ac.uk}} \author{
John Armstrong\thanks{Email: \href{mailto:john.armstrong@kcl.ac.uk}{john.armstrong@kcl.ac.uk}}
\and
Purba Das\thanks{Email: \href{mailto:purba.das@kcl.ac.uk}{purba.das@kcl.ac.uk}}
}
\date{Department of Mathematics, King's College London}
\begin{document}

\maketitle
\begin{abstract}
We show how the robustness of gamma hedging can be understood just using the concepts of $p^{th}$ variation along a partition sequence and Taylor's theorem. There is no need to define an integral or proving a version of It\^o's lemma. The same approach allows classical results on delta-hedging to be proved without defining an integral and without the need to define the concept of self-financing in continuous time. We show that the approach can also be applied to barrier options
and Asian options. Finally, we extend the framework to path-dependent derivatives. We introduce a Gubinelli-type control along a sequence of partitions, and use it to define a pathwise notion of arbitrage.  Finally, we show that in any arbitrage-free pricing model the greeks must satisfy a linear relation. If one additionally assumes time and scaling symmetries, this yields the Black--Scholes PDE.
\end{abstract}

\tableofcontents

\section*{Introduction}
The classical delta-hedging strategy enables European options to be replicated under the assumption that the underlying asset follows an a priori known probability model. In the delta-hedging strategy, one chooses at each time a portfolio of hedging instruments so that the first-order sensitivity of the portfolio to changes in asset prices (the delta) matches the delta of the asset one seeks to replicate. In the gamma-hedging strategy, one ensures additionally that the second-order sensitivity (the gamma) is matched.

In \cite{armstrongIonescu} it is shown using rough-path theory that the discrete-time gamma-hedging strategy allows European options to be replicated with arbitrary accuracy on the assumption that the underlying asset has finite $p$-variation  for $p<3$ and that there exist hedging options whose prices are given by a function which satisfies the Black--Scholes partial differential equation. This result does not require any probabilistic model for the underlying asset.

This paper gives a more elementary
account of the success of gamma hedging that eliminates the need to use rough-path integration theory. This dramatically simplifies the exposition to a single application of Taylor's theorem.

The central trick in the argument is to express the profit and loss of the discrete-time hedging strategy in a single expression, which then tends to zero {\color{red} as the time between rebalancing goes to zero}. If one instead tries to show that sub-expressions are equal and opposite, as is natural when trying to show that one strategy replicates a payoff, the cancellations become obscure. Moreover, the individual terms do not necessarily converge unless they are correctly renormalised, and it is this issue which introduces the complexities into rough-path theory. There is a parallel notion of outer measure, originally introduced by Vovk to develop a model-free approach for optimal hedging strategies \cite{lochowski2018}.

The strategy of proving convergence of discrete-time hedging path by path goes back to F\"ollmer's calculus without probabilities \citep{follmer}. It has been developed in several directions.  \cite{bickWillinger} considered price paths whose quadratic variation along the partition sequence equals $\int_0^t\sigma^2(u,S_u)S_u^2\,du$. They showed that the delta hedge of a European claim replicates the claim exactly whenever the price function solves the corresponding Black--Scholes-type PDE. The gains from trading are given by F\"ollmer's pathwise integral. The strategy is self-financing in a pathwise continuous-time sense. \cite{tikanmaki2013} used the same pathwise It\^o formula to study robustness of delta hedging. The hedge works for any model sharing a given quadratic variation. Mixed fractional Brownian motion is one example. \cite{riga2016} and \cite{chiu2023} extended the framework to path-dependent claims. They build a general pathwise theory of continuous-time trading on the functional It\^o calculus of Dupire and Cont--Fourni\'e \citep{cont2010}. Self-financing strategies are defined as pathwise limits of discrete rebalancing schemes. Gains are F\"ollmer integrals of non-anticipative functionals. Replication follows from a functional It\^o formula together with a path-dependent PDE. \cite{schied2016} studied the class of delta-hedging strategies for paths with prescribed quadratic variation. They proved that this class admits no pathwise arbitrage.

These works have a number of common features that distinguish them from this paper. They all concern delta hedging, so the quadratic variation of the price path must exist and coincide with the one used to price the hedging instruments. They all construct a pathwise integral and prove a pathwise or functional It\^o formula. Their replication statements are about a continuous-time strategy whose value is identified with the claim. This requires a pathwise notion of self-financing.

By contrast, our main result is Theorem~\ref{thm:mainTheorem} which concerns gamma hedging, does not require finite quadratic variation and requires no relation between the volatility used to price the hedging options and the price path. This gives an elementary proof of the robustness result of \cite{armstrongIonescu}. That result was obtained with rough-path theory. Our proof consists of a Taylor expansion of the discrete-time profit and loss followed by a direct estimate. We do not define an integral or prove an It\^o formula, and we do not introduce a continuous-time strategy. The theorem is a statement about the discrete-time profit and loss itself. By using a vector-valued 1-dimensional Taylor expansion, our result generalises the delta-hedging results using F\"ollmer's model-free approach to the gamma-hedging setting without relying on Rough-path theory or the F\"ollmer--It\^o formula. In addition, we believe that our use of pathwise no-arbitrage arguments to derive a PDE for prices is a new contribution.


We begin by examining the case of gamma-hedging European options, but will later show how our results can be extended to path-dependent options. Our results for European options rely on PDE techniques and so can be trivially extended to barrier options.
We will also show how our results can be applied directly to Asian options. More theoretically, we will describe a framework which explains how our techniques can be extended to fully path-dependent options subject to the existence of a second-order Gubinelli-type derivative. This definition makes the connection between our approach and rough-path theory explicit.

This theory of replication for path-dependent derivatives will allow us to define a pathwise 
concept of arbitrage. We show how the Black-Scholes PDE can be derived by a no-pathwise-arbitrage argument. We also show that derivative prices in a pathwise arbitrage-free model will under modest conditions satisfy a path-dependent PDE.

If one models a stock price as the exponential of drifted Brownian motion, one can describe the stock price trajectories in the Black--Scholes model without needing to introduce stochastic differential equations. Thus our results make it
possible to explain the theory of replication in the Black--Scholes model using only the existence of Brownian motion and the fact that it almost surely has finite quadratic variation.

Our results also allow us to derive the Black--Scholes PDE by a no-arbitrage argument assuming only the existence of
some exchange-traded derivative and the symmetry of the no-arbitrage condition under translations and dilations. This gives a fully non-probabilistic derivation of the Black--Scholes formula. The same symmetry is assumed in the classical derivation when one assumes that the dynamics of the stock is given by geometric Brownian motion.

In summary, we show that there is no need to develop a continuous-time theory to prove the convergence of discrete-time hedging or to derive the Black--Scholes PDE by a form of no-arbitrage argument. These are the central novelties in our argument.
As a result of this approach, all the technicalities of developing continuous-time theories therefore disappear. The limitation is that continuous-time theories have many applications in financial mathematics beyond hedging, Merton's problem being the most classical example. Thus, while it is possible to understand hedging without ever defining continuous-time strategies, the concept of continuous-time strategies remains extremely useful.  Thus, the theories of It\^o calculus, rough-path calculus and functional It\^o calculus are still enormously valuable in wider financial mathematics applications.

\subsubsection*{AI statement}

AI was used for proof-reading the final draft of this paper.

\section{Pathwise convergence}

In this section, we give a short proof of the convergence of the gamma-hedging strategy for European options without any stochastic integration theory. For simplicity, we assume the risk-free rate is $r=0$. The general case can be deduced from this by a change of num\'eraire.

Let $\pi = (\pi_N)_{N=1}^\infty$ be a sequence of partitions of
$[0,T]$ with maximum step size, denoted $\|\pi_N\|$, going to zero. So $\pi_N = \{ t_0^N, \ldots, t_{C(\pi_N)}^N \}$
with $t_0^N = 0 < \ldots < t_k^N < \ldots < t_{C(\pi_N)}^N = T$
and $\|\pi_N\|:=\sup_{1 \leq i \leq C(\pi_N)} ( t_i^N-t^N_{i-1} )$, where $C(\pi_N)+1$ represents the total number of partition points of $\pi_N$.
In our next definition, we write $[u, v] \in \pi_N$ to indicate that $u$ and $v$ are both in $\pi_N$ and are immediate successors. We will use this convention throughout.

\begin{definition}
Let $p \in {\mathbb R}_{\geq{1}}$ and $({\cal V} {\color{red} , \|\cdot\|)}$ be a real normed vector space. A continuous path $X \in C([0,T], {\cal V})$ is said to have {\em vanishing $p$\textsuperscript{th}-variation} along a sequence of partitions $\pi = (\pi_N)_{N \geq 1}$ if
\[
\lim_{N \to \infty} \sum_{[u,v] \in \pi_N} \|X_v-X_u\|^p =0.
\]
\end{definition}

The term ``vanishing $p$\textsuperscript{th}-variation'' is designed
to be consistent with the use of the term $p$\textsuperscript{th}-variation in \cite{contBook2016}. However, one should be careful not to confuse this with the very similarly named concept of $p$-variation defined for ease of comparison below.  

\begin{definition}
Let $p \in {\mathbb R}_{\geq{1}}$ and $({\cal V} {\color{red} , \|\cdot\|)}$ be a real normed vector space. Given a path $X \in C([0,T],{\cal V})$ we define
the $p$-variation of $X$ to equal
\[
\| X \|_{\pvar} = \left( \sup_{\pi \in {\cal P}[0,T]} \sum_{[u,v] \in \pi} \|X_v-X_u\|^p \right)^{\frac{1}{p}}.
\]
where ${\cal P}[0,T]$ is the set
of all finite partitions of $[0,T]$.
\end{definition}

Notice that the concept of $p$-variation is a supremum over all partitions, whereas when discussing $p$\textsuperscript{th}-variation it is always computed along a given sequence $\pi$. Keeping this in mind helps one distinguish these similarly named concepts more easily. We will write $\| X \|_{\pvar;[s,t]}$ for the $p$-variation defined over other intervals $[s,t]$ in a similar way.

Let us describe the multi-dimensional calculus notation we will be using. Given a finite-dimensional {\color{red}inner-product} space $({\cal V} {\color{red} , \langle \cdot,\cdot\rangle_{\cal V}})$, a second vector space ${\cal W}$ and a function $F:{\cal V} \to {\cal W}$, we will write $\nabla F: {\cal V} \to {\cal V} \otimes {\cal W} $ for the gradient of $F$. Given $X \in {\cal V}$, we write  $\nabla_X F := \langle \nabla F, X\rangle_{\cal V}$. When ${\cal V} = {\cal V}^1 \oplus {\cal V}^2$ we write ${\nabla^1} F$ and $\nabla^2 F$ for the components of $\nabla F$ in ${\cal V}^1$ and ${\cal V}^2$ respectively, and define $\nabla^i_X F:= \langle \nabla^i F, X \rangle_{\cal V}$ for $X \in {\cal V}$.

{\color{red}
Throughout the remainder of the paper, we will assume that $\pi$ is a sequence
of partitions of $[0,T]$ and ${\cal V}^\alpha$
with $\alpha \in \{1, 2\}$ are finite-dimensional inner-product spaces.
We will assume that $U \subseteq {{\cal V}^1 \oplus {\cal V}^2}$ is an open subset.
We will write ${\cal A}_{\pi}$ for the set of paths $X \in C([0,T], U)$
where the component $X^\alpha$ taking values in ${\cal V}^\alpha$ has vanishing $p_\alpha$\textsuperscript{th} variation along $\pi$ for some $1 \leq p_1\leq2$, $1 \leq p_2\leq3$ and
$\frac{1}{p_1}+\frac{1}{p_2} \geq 1$.
We denote $C^{2+}$ to be the class of $C^2$ functions $F:U \to {\mathbb R}$ whose second derivatives are locally Lipschitz continuous.
\begin{definition}
A {\em pricing function} is a $C^{2+}$ function $F:U \to {\mathbb R}$.
\end{definition}
}
\noindent Write $[0,T)^\pi : = \left(\bigcup_{N=1}^\infty  \pi_N \right)\cap [0,T)$ for the set of partition points in $[0,T)$. Let $q^i$ ($i=0,\ldots, n$) be a bounded function $q^i:\bigcup_{N=1}^\infty \pi_N\to \R$.

{\color{red} For each $N$, we use $q^i$ to obtain a discrete-time self-financing trading strategy where
a quantity $q^i_u\in\R$ of an asset with price given by $F^i$ is purchased at each time $u \in \pi_N$.  
Equation \eqref{eq:pnl} in our next Theorem then gives the total profit and loss of this trading strategy $[0,T]$. Since we insist $q^0_u=-1$, this strategy can be interpreted as a replication strategy for asset $0$ using the hedging instruments with prices $F^i$. We impose linear conditions on the first and second derivatives of the pricing functions $F^i$. When we apply the theorem in concrete settings, these can be interpreted as matching greeks of the asset to be replicated with the greeks of the hedging instruments.}

\begin{theorem}
\label{thm:mainTheorem}
\par For indices $i \in \{0,1, \ldots,  n\}$, 
let $F^i: U \to \R$ be {\color{red} pricing functions}. Let $q^i$ ($i=0,\ldots, n$) be a bounded function {\color{red} $q^i: [0,T)^\pi \to \R$
and let $X \in {\cal A}_{\pi}$.} Suppose that at all times $t \in [0,T)^\pi$,
\begin{align}
&\sum_{i=0}^n q^i_t \nabla_V^\alpha F^i (X^1_t,X^2_t)  = 0, \quad \forall\, \alpha \in \{1,2\}, \, V \in {\cal V}^\alpha
\label{eq:firstOrder}\\
&\sum_{i=0}^n q^i_t \nabla^2_{V_1} \nabla^2_{V_2}  F^i (X^1_t,X^2_t)  = 0,  \, \quad \forall\, \, V_1, V_2 \in {\cal V}^2,
\label{eq:secondOrder}
\end{align}
then defining
\begin{equation}
P_N:= \sum_{[u,v] \in \pi_N} \sum_{i=0}^n   q^i_u( F^i(X^1_v, X^2_v)-F^i(X^1_u, X^2_u))
\label{eq:pnl}
\end{equation}
and
\[
\bar{P}_N:= \sum_{[u,v] \in \pi_N} {\color{red}\Bigg|} \sum_{i=0}^n   q^i_u( F^i(X^1_{\color{red}v}, X^2_{\color{red}v})-F^i(X^1_{\color{red}u}, X^2_{\color{red}u})) {\color{red}\Bigg|}
\]
we have
\begin{equation}
\lim_{N\to \infty} P_N = 0
\label{eq:vanishingPnL}
\end{equation}
and
\begin{equation}
\lim_{N\to \infty} \bar{P}_N =0.
\label{eq:vanishingAbsPnL}
\end{equation}
\end{theorem}

\begin{proof}
We define an indexing set 
\[
{\cal I}=\{(a_1,a_2): a_1 \in \{0,1,2\}, a_2 \in \{0,1,2\}, 0<
a_1+a_2 \leq 2 \},
\]
which we will use to label terms in a Taylor expansion. For each interval $[u,v]$ in the partition $\pi_N$, we will perform a Taylor expansion of $F^i$ at $(X^1_u,X^2_u)$ with the integral form of the Taylor remainder (See for example, \cite{perkowski2019}). So we will have,
\begin{align}\label{eq:taylorExpansion}
&
\sum_{[u,v]\in\pi_N}{\color{red}\Big|} \sum_{i=0}^n q_u^i\bigl(F^i(X_v^1,X_v^2)-F^i(X_u^1,X_u^2)\bigr){\color{red}\Big|}
\nonumber \\ &\leq \sum_{[u,v]\in\pi_N} 
{\color{red}\Bigg|}\sum_{\substack{(a_1,a_2)\in\mathcal I}}\sum_{i=0}^n\frac{1}{a_1!\,a_2!}q_u^i(\nabla^1_{X^1_v-X^1_u})^{\otimes a_1}(\nabla^2_{X^2_v-X^2_u})^{ \otimes a_2}F^i(X_u^1,X_u^2) {\color{red}\Bigg|} \nonumber\\
& \qquad + {\color{red}
\sum_{[u,v]\in\pi_N} \Bigg|\sum_{i=0}^n q_u^i  \int_0^1 (1-\lambda) 
\sum_{a_1+a_2=2} \frac{2!}{a_1!\,a_2!}
\Big[(\nabla^1_{X^1_v-X^1_u})^{\otimes a_1} (\nabla^2_{X^2_v-X^2_u})^{\otimes a_2} F^i\bigl(\gamma_{u,v}(\lambda)\bigr)}\nonumber
\\&\qquad\hspace{6cm}{\color{red}-(\nabla^1_{X^1_v-X^1_u})^{\otimes a_1} (\nabla^2_{X^2_v-X^2_u})^{\otimes a_2} F^i(X_u^1,X_u^2)\Big] d\lambda \Bigg|
},
\end{align}
where, 
\begin{equation}
\gamma_{u,v}(\lambda) = (X_u^1,X_u^2) + \lambda(X^1_v-X^1_u, X^2_v-X^2_u), \; \forall \lambda\in[0,1].
\label{eq:defLowerCaseGamma}
\end{equation}
Since the image of $X$ is compact and $U$ is open, there is a compact $K\subseteq U$ containing a
neighbourhood of the path, and since $\|X_v-X_u\|\le\operatorname{osc}(X,\|\pi_N\|)\to0$, for $N$
large every segment $\gamma_{u,v}([0,1])$ lies in $K$. Let $L$ be a
common Lipschitz constant of the second derivatives of the $F^i$ on $K$.
So for all $\lambda\in[0,1]$, we have the bound for all $a_1+a_2 = 2$
\begin{align}
\big|(\nabla^1_{X^1_v-X^1_u}&)^{\otimes a_1} (\nabla^2_{X^2_v-X^2_u})^{\otimes a_2} F^i\bigl(\gamma_{u,v}(\lambda)\bigr)
-(\nabla^1_{X^1_v-X^1_u})^{\otimes a_1} (\nabla^2_{X^2_v-X^2_u})^{\otimes a_2} F^i(X_u^1,X_u^2)\big| \nonumber\\
&\leq
\big\|(\nabla^1)^{\otimes a_1}(\nabla^2)^{\otimes a_2}F^i\bigl(\gamma_{u,v}(\lambda)\bigr)
-(\nabla^1)^{\otimes a_1}(\nabla^2)^{\otimes a_2}F^i(X_u^1,X_u^2)\big\|\,
\|X^1_v-X^1_u\|^{a_1}\,\|X^2_v-X^2_u\|^{a_2}\nonumber\\
&\leq L\,\lambda\,\|X_v-X_u\|\,\|X^1_v-X^1_u\|^{a_1}\|X^2_v-X^2_u\|^{a_2}.
\label{eq:lipEstimate}
\end{align}
The terms
\[
\big\|(\nabla^1)^{\otimes a_1}(\nabla^2)^{\otimes a_2}F^i\bigl(\gamma_{u,v}(\lambda)\bigr)
-(\nabla^1)^{\otimes a_1}(\nabla^2)^{\otimes a_2}F^i(X_u^1,X_u^2)\big\|, \qquad \forall \; (a_1,a_2)\in \mathcal I \; \; \text{and, } \forall \lambda\in[0,1]
\]
are bounded by our assumption of Lipschitz continuity of second derivatives and the compactness of the interval.

We now see that for each $(a_1,a_2) \in {\cal I}$ the terms on the right of equation \eqref{eq:taylorExpansion} vanish in the limit due to our various assumptions. When $a_1+a_2=2$, the terms vanish in the limit due to the bounds on $q^i$, the bounds on the derivative terms, and the fact that $(X^1,X^2)$ has vanishing $p$-th variation along $\pi$ for $p=3$. When $a_1=0$ and $a_2=2$, the term vanishes due to \eqref{eq:secondOrder}. When $a_1=1$ and $a_2=1$ the term vanishes in the limit as
\[
\lim_{N \to \infty} \sum_{[u,v] \in \pi_N} \|(X^1_v - X^1_u)\|\| X^2_v - X^2_u \| = 0
\]
by H\"older's inequality. When $a_1=2$ and $a_2=0$, the term vanishes as $X^1$ has vanishing $p$-th variation along $\pi$ for $p=2$. When $a_1=1$ and $a_2=0$ or $a_1=0$ and, $a_2=1$ the term vanishes by \eqref{eq:firstOrder}. This exhausts the possibilities for $(a_1,a_2) \in {\cal I}$.
\\Hence equation \eqref{eq:vanishingAbsPnL} holds. Equation \eqref{eq:vanishingPnL} therefore holds a fortiori.
\end{proof}

{\color{red} 
With the understanding that we do not define any notion of continuous-time trading strategy, we note that one can choose an arbitrary extension of $q^i$ to the whole interval $[0,T]$ and so we will henceforth view $q^i$ as a map defined on the whole interval.
}

{\color{red}Before we apply Theorem \ref{thm:mainTheorem}, we make a number of definitions.
We define a function $\rho$ on ${\mathbb R} \times [0,T) \times {\mathbb R} \times (0,\infty)$ by
\begin{equation}
\rho(Z;t,z,\sigma)=\frac{1}{\sigma \sqrt{2\pi(T-t)}}
\exp\left(-\frac{\left(Z - z  + \frac{1}{2}\sigma^2(T-t)\right)^2}{2 \sigma^2(T-t)}
\right).
\label{eq:defQ}
\end{equation}
It represents a logarithmic transform of the pricing kernel in the Black--Scholes model. We define the Black--Scholes price of a European derivative with payoff function $f:(0,\infty) \to {\mathbb R}$ at time $T$ by 
\begin{equation}
BS(f,T,\sigma,t,s)=\int_{-\infty}^\infty \rho(Z;t,\log(s),\sigma) f(\exp(Z))\, dZ.
\label{eq:greensFunction}
\end{equation}
\begin{remark}
We do not need It\^o calculus to prove our results, but we note that if $(S_u)_{u \in [t,T]}$ is a stochastic process satisfying
\begin{equation}
d S_u = \sigma S_u d {W}_u, \quad S_{\color{red}t}=s
\label{eq:geometricBrownian}
\end{equation}
for a Brownian motion $W$, then
\[
BS(f,T,\sigma,t,s)={\mathbb E}({\color{red}f(S_T)}).
\]
This can be proved by taking $z_u:=\log(S_u)$ and applying It\^o's Lemma to show that $z$ is a drifted Brownian motion. Equation \eqref{eq:defQ} then gives the probability density function for $Z=z_T$.
\end{remark}
}

Let us now state a first application of Theorem \ref{thm:mainTheorem}: so long as a stock price path $X^2:[0,T]\to {\mathbb R}$ has vanishing $p$-th variation along $\pi$ for some $p<3$, and so long as one can purchase appropriate hedging options at the Black--Scholes price, then it is possible to replicate a European option to arbitrary accuracy by using the gamma-hedging strategy. This is made precise
in the corollary below.
\begin{corollary}
\label{cor:gammaHedging}
Let $X^2 \in C([0,T], \R_{> 0})$ be a path with vanishing $p$\textsuperscript{th} variation {\color{red} along a sequence of partitions} $\pi$ for some $p\leq 3$. Let $\sigma>0$ be given and let the risk-free interest rate $r=0$. Let $f^i:{\color{red}[0,\infty)}\to \R$ ($i=0,\ldots,n)$ be {\color{red}measurable functions of at most polynomial growth} and let $F^i$ be given by the prices of {\color{red}European options with these payoffs} in the Black-Scholes model, so $F^i(t,s):=BS(f^i, T, \sigma, t, s)$.
If $q^i$ ($i=0,\ldots, n$) are bounded functions $q^i:[0,T]\to \R$ satisfying for all $t\in [0,T)$,
\begin{equation}
\sum_{i=0}^n q^i_t \frac{\partial F^i}{\partial s} (t, X^2_t) = 0
\label{eq:deltaHedge}
\end{equation}
\begin{equation}
\sum_{i=0}^n q^i_t \frac{\partial^2 F^i}{\partial s^2} (t, X^2_t) = 0
\label{eq:gammaHedge}
\end{equation}
with $q^0=-1$ then
\begin{equation}
\label{eq:replicationPathwise}
f^0(X_T) = F^0(0,X^2_{\color{red}0}) + \lim_{N \to \infty} \sum_{[u,v]\in \pi_N} \sum_{i=1}^n  q^i_u( F^i(v,X^2_v)-F^i(u, X^2_u)).
\end{equation}
{\color{red} so long as each $f^i$ is continuous  in a neighbourhood of $X_T$.}
\end{corollary}
\begin{proof}
{\color{red}
To begin assume each $f^i$ is smooth. 
\par Our assumption that the $f^i$ have at most polynomial growth ensures that $f^i(\exp(Z))$ has at most exponential growth in $Z$. Together with the quadratic-exponential decay of $\rho$, this ensures by dominated convergence that we can differentiate equation \eqref{eq:greensFunction} under the integral sign. 
Moreover, since function $\rho(Z;t,z)$ satisfies the equation
\[
\frac{\partial \rho}{\partial t}(Z;t,z)
-\frac12 \sigma^2 \frac{\partial \rho}{\partial z}(Z;t,z)
+ \frac{1}{2} \sigma^2 \frac{\partial^2 \rho}{\partial z^2}(Z;t,z)
 = 0
\]
by equation \eqref{eq:greensFunction}
and a change of variables in the integral,
the $F^i$ satisfy the Black--Scholes PDE
\begin{equation}
\frac{\partial F^i}{\partial t}(t,s) 
+ \frac{1}{2} \sigma^2 s^2 \frac{\partial^2 F^i}{\partial s^2}(t,s)
 = 0.
\label{eq:blackScholesPDE}
\end{equation}

Now define $X^1:[0,T]\to \R$ by $X^1_t=t$.
The Black--Scholes PDE defines a linear
relationship between the components of the vector
\[
D_t:=\left(\sum_{i=0}^n q^i_t\frac{ \partial F^i}{\partial t}(t,X^2_t),\;
\sum_{i=0}^n q^i_t\frac{\partial F^i}{\partial s}(t,X^2_t),\;
\sum_{i=0}^n q^i_t\frac{\partial^2 F^i}{\partial s^2}(t,X^2_t)\right)\in\R^3 .
\]
The same is true of equations \eqref{eq:deltaHedge} and \eqref{eq:gammaHedge}.
These equations are linearly independent because the Black--Scholes equation is the only one involving the partial derivative with respect to time; the other two are manifestly independent. With three linearly independent equations on a three-dimensional vector space, we see that all the components of $D$ must vanish. Hence
\eqref{eq:firstOrder} and \eqref{eq:secondOrder} must hold.
}
Thus equation \eqref{eq:vanishingPnL} must hold. Since $q^0=-1$ we can collapse the telescoping sum in $F^0$ to obtain
equation \eqref{eq:replicationPathwise}.

Now for the case when $f^i\not \in C^\infty$, $F^i$ will still be smooth for all times $t<T$ and so the result holds on any interval $[0,T-\epsilon)$
when we consider derivatives with payoffs $F^i(T-\epsilon,s)$ for all $0<\epsilon<T$. Since $f^0$ is continuous at $X^2_T$ then, by polynomial growth and dominated convergence in \eqref{eq:greensFunction}, $F^0(T-\epsilon,X^2_{T-\epsilon})\to f^0(X^2_T)$ as
$\epsilon\to 0$. By continuity, the result will also hold on $[0,T]$.
\end{proof}

The argument of the corollary can be generalised straightforwardly to prove
the convergence of the gamma-hedging strategy for European derivatives in any diffusion models which are sufficiently regular for us to be able to establish the necessary smoothness of the prices. One simply replaces the Black--Scholes PDE with the Feynman--Kac PDE. However, our next corollary is specific to the Black--Scholes model. It
shows that the Black--Scholes gamma-hedging strategy converges so long as the underlying asset and the implied volatility process are sufficiently regular.

\begin{corollary}
Let $X^2 \in C([0,T], \R_{>0})$ be a path with vanishing $p_2$\textsuperscript{th} variation along $\pi$. Let $\sigma^{{\color{red}IV}} \in C([0,T], \R_{> 0})$ be a path of vanishing $p_1$\textsuperscript{th} variation along $\pi$. Suppose $p_1\leq 2$, $p_2\leq 3$ and $\frac{1}{p_1}+\frac{1}{p_2} \geq 1$. Let $f^i: \R_{> 0} \to \R$ ($i=0,\ldots,n)$ be continuous functions with at most polynomial growth, representing the payoff  for European options with maturity $T$. Suppose that $F^i(t,\sigma,s):=BS(f^i, T, \sigma, t, s)$ are $C^{2+}$ except for the final time $T$.
If $q^i$ ($i=0,\ldots, n$) are bounded functions $q^i:[0,T]\to \R$ satisfying for all $t\in[0,T)$
\begin{equation}
\sum_{i=0}^n q^i_t \frac{\partial F^i}{\partial s} (t,\sigma^{{\color{red}IV}}_t,X^2_t) = 0
\label{eq:deltaHedge2}
\end{equation}
\begin{equation}
\sum_{i=0}^n q^i_t \frac{\partial^2 F^i}{\partial s^2} (t,\sigma^{{\color{red}IV}}_t, X^2_t) = 0
\label{eq:gammaHedge2}
\end{equation}
with $q^0=-1$ then
\begin{equation*}
f^0(X_T) = F^0(0,\sigma^{{\color{red}IV}}_{\color{red}0}, X^2_{\color{red}0}) + \lim_{N \to \infty} \sum_{[u,v]\in \pi_N} \sum_{i=1}^n  q^i_u( F^i(v,\sigma^{{\color{red}IV}}_v, X^2_v)-F^i(u,\sigma_u^{\color{red}IV}, X^2_u)).
\end{equation*}
\end{corollary}
\begin{proof}
{\color{red}
We may assume that  $f^i$ is smooth and hence $F^i$ is also smooth; when $f^i$ is not smooth, we follow the same method as in Corollary \ref{cor:gammaHedging}. Differentiating the expression for $\rho$
in equation \eqref{eq:defQ} and useing the  change of variables $z=\log s$  to establish the identity
\[
\frac{\partial\rho}{\partial\sigma} (t,\sigma, \log s)
=\sigma(T-t) \left(\frac{\partial^2\rho}{\partial z^2} (t,\sigma, \log s)  -  \frac{\partial\rho}{\partial z} (t,\sigma, \log s)\right)
\]
It follows that the price of 
European derivatives with maturity $T$ 
in the Black--Scholes model satisfy
\begin{equation}
\sigma(T-t) {s}^2 \frac{\partial^2 F^i}{\partial {s}^2} (t,\sigma, {s}) = \frac{\partial F^i}{\partial \sigma} (t,\sigma, {s})
\label{eq:volIdentity}
\end{equation}
in addition to the Black--Scholes PDE.

Now take $X^1_t=(t,\sigma^{IV}_t)$.
The Black--Scholes equation, equation \eqref{eq:volIdentity}, equation \eqref{eq:deltaHedge2} and \eqref{eq:gammaHedge2} gives us 4 linearly independent equations in the four partial derivatives featuring in equations \eqref{eq:firstOrder} and \eqref{eq:secondOrder}. It follows that the last two equations hold, and our result follows.
}
\end{proof}

It is natural to ask when such $\{q^i\}_{i=0,\cdots,n}$ exist. If $f^1(X^2)=X^2$ (representing the case when the underlying is traded) and if $f^2$ is convex and non-linear (which is the case for standard put and call options), then we will have $\nabla^2\nabla^2 F^2(t,X^2)>0$ whenever $t<T$ (see \cite{armstrongIonescu} for the proof). Since $\nabla^2\nabla^2 F^1(t,X^2)=0$, the vectors $(\nabla^2 F^i, \nabla^2 \nabla^2 F^i)(t,X^2)$ will be linearly independent for all $t<T$. This means that we can always find, possibly unbounded, $q^i$ satisfying equations 
\eqref{eq:deltaHedge2} and \eqref{eq:gammaHedge2}.
As a result if the stock and a single put or call are traded at the Black--Scholes price we can replicate any option $f^0$ using
the gamma-hedging strategy up to a time $T^\prime<T$, but the $q^i$ need not remain bounded as $t\to T$. If the payoff of the hedging instrument has a non-vanishing second derivative, the $q^i_t$ will remain bounded
unless the stock price path ends at a non-smooth point of $F^0$.


{\color{red}
The key to our results is that  in Theorem \ref{thm:mainTheorem} we have taken the sum over all $i=0 \ldots n$ and shown that the result equals zero. If one attempts to rearrange the formulae so that the $i=0$ terms appear on the left and the $i=1 \ldots n$ terms appear on the right then both sides of the equality will in general diverge. This is addressed in proofs using It\^o calculus by using convergence in probability to ensure the integrals converge. In proofs using rough-path calculus, the F\"ollmer integral or pathwise It\^o calculus, appropriate integrals can be defined using a renormalization approach.
We never attempt to make sense of any equality between the $i=0$ and the $i=1 \ldots n$ terms. This is why we do not need to develop any continuous-time theory and is why most of the technicalities required in other approaches to replication disappear.
}

\section{Paths with finite quadratic variation}

\begin{definition}
\label{def:quadraticVariation}
A continuous path $X \in C([0,T], {\cal V})$ is said to have a finite \emph{quadratic-variation} along a sequence of partitions $\pi$ with mesh tending to zero if the sequence of ${\cal V}^{\otimes 2}$-valued measures 
\[
\mu^n := \sum_{[u,v] \in \pi_n} \delta(\cdot - u) (X_v - X_u)^{\otimes 2}
\]
converges weakly to a symmetric measure $\mu$ without atoms. In that case we write $[X]_t := \mu([0,t])$ for $t \in [0,T]$, and we call $[X]$ the \emph{quadratic-variation} of $X$ along the partition sequence $\pi$.
\end{definition}

The following lemma gives a simple characterization of this property.
\begin{lemma}
Let $X \in C([0, T],{\cal V})$. $X$ has finite quadratic variation along $\pi$ if and only if there exists a continuous
function $[X]$ such that
\begin{align}\label{def: p-th var (cont)}
   \forall t \in [0, T],
\sum_{\substack{[u,v]\in \pi_n:\\
u\leq t}}
(X_v - X_u)^{\otimes 2} \xrightarrow[]{n\to \infty} [X]_t. 
\end{align} 
If this property holds, then the convergence in \eqref{def: p-th var (cont)} is uniform.    
\end{lemma}

The point of Definition \ref{def:quadraticVariation} is that it ensures that if $Y^*_u$ is a piecewise continuous path in ${\cal V}^*\otimes {\cal V}^*$ then we will 
have
\begin{equation}
\lim_{n \to \infty} \sum_{[u,v]\in \pi_n} Y^*_u (X_v-X_u)^{\otimes 2}
= \int_0^T Y^*_u d[X]_u = \lim_{n \to \infty} \sum_{[u,v]\in\pi_n} Y^*_u ([X]_v - [X]_u).
\label{eq:quadraticVariationEquivalence}
\end{equation}
This is proved in \cite[Proposition~5.3.5]{contBook2016}.

\begin{note}
To translate between the notation of \cite[Proposition~5.3.5]{contBook2016} and our notation, one first chooses a basis $\{e_1, \ldots, e_n\}$ for ${\color{red}{\mathcal V}}$ with its corresponding dual basis $\{e^1, \ldots, e^n\}$. An arbitrary element  {\color{red} $a \in \mathcal V^* \otimes \mathcal  V^*$} may be written as $\sum_{i,j} a_{ij} e^i \otimes e^j$ and an element  $b \in {\color{red}{\mathcal V}} \otimes {\color{red}{\mathcal V}}$ may be written as $\sum_{i,j} b^{ij} e_i\otimes e_j$.
The value of the pairing $a \cdot b = \sum_{i,j} a_{ij} b^{ij} = \mathrm{tr} (A B^\top)$ where $A$ is the matrix with components $a_{ij}$ and $B$ is the matrix with components $b^{ij}$. The transpose is unnecessary if $b$ or $a$ is symmetric.
\end{note}

We would like to prove an analogue of Theorem \ref{thm:mainTheorem} but for paths with finite quadratic variation along $\pi$. Before we
state this result, we will need a preliminary
lemma that allows us to discretize ODEs involving
the quadratic variation.

Given a path $X:[0,T]\to {\cal V}$, we define the modulus of continuity
\[
\osc(X,\delta):=
\sup \{ \|X_s-X_t\| \mid s,t \in [0,T], |s-t|\leq \delta \}.
\]
Note that
$\osc(X, \delta) \to 0$ as $\delta \to 0$ {\color{red}if and only if} $X$ is uniformly continuous.

It is known that given a piecewise continuous function $f:[0,T]\to{\cal V}^*$ and a function of bounded $1$-variation $g:[0,T] \to {\cal V}$, the Riemann--Stieltjes
integral $\int_0^t f_u \, dg_u$
is well-defined for all $t \in [0,T]$.

\begin{lemma}
\label{lemma:riemannSums}
Let $f^i$ be piecewise continuous functions and $g^i$ be continuous functions with bounded $1$-variation. Suppose that an ordinary integral equation of the form $\sum_i \int_0^t f^i_u \, d g^i_u = 0$ for all $t \in [0,T]$ holds and that $y:[0,T]\to \R$ is bounded, then
\[
\left|
\sum_{ [u,v] \in \pi_N}
\sum_{i=1}^n 
y_u f^i_u (g^i_v - g^i_u)
\right|
\leq
\| y \|_\infty
\sum_{i=1}^n \osc( f^i, \| \pi_N\|) \|g^i\|_{1\var}.
\]
The right-hand side tends to zero
as $N \to \infty$.
\end{lemma}
\begin{proof}
Without loss of generality, we will assume $f^i$ is a continuous path, as the contribution from at most finitely many jump points goes to $0$ as $N\to \infty$. Recall \cite[Chapter 2, Proposition 2.2]{frizVictoir} that if $f$ is continuous and 
$g$ has finite $1$-variation then we have the following estimate
for the Riemann-Stieltjes integral
\[
\left|\int_u^v f_s \, dg_s \right| \leq (\sup_{s \in [u,v]} |f_s| ) \| g \|_{1\var;[u,v]}.
\]
We now compute that
\begin{align*}
\left|
\sum_{ [u,v] \in \pi_N}
\sum_{i=1}^n 
y_u f^i_u (g^i_v - g^i_u)
\right|
&=
\left|
\sum_{ [u,v] \in \pi_N}
\sum_{i=1}^n 
\left(
y_u f^i_u (g^i_v - g^i_u)
- y_u \int_u^v f^i_s \, d g^i_s 
\right)
\right| \\
&\leq
\| y \|_\infty
\sum_{ [u,v] \in \pi_N}
\sum_{i=1}^n 
\left|
f^i_u (g^i_v - g^i_u)
- \int_u^v f^i_s \, d g^i_s 
\right| \\
&\leq
\| y \|_\infty
\sum_{ [u,v] \in \pi_N}
\sum_{i=1}^n 
\left|
\int_u^v (f^i_u-f^i_s) \, d g^i_s 
\right| \\
&\leq
\| y \|_\infty
\sum_{ [u,v] \in \pi_N}
\sum_{i=1}^n 
\osc(f^i, \| \pi_N \|) \| g^i \|_{1\var;[u,v]} \\
&=
\| y \|_\infty
\sum_{i=1}^n 
\osc(f^i, \| \pi_N \|) \| g^i \|_{1\var}.
\end{align*}
\end{proof}

Given this Lemma we now prove the main result of this section. 


\begin{theorem}
Let $\pi$ be a {\color{red}sequence of} partitions with mesh tending to 0.
Suppose that $X \in C([0,T], {\cal V})$ has finite quadratic variation
along $\pi$
and that ${\color{red}{(\cal V}, \langle\cdot,\cdot\rangle)}$ is a finite-dimensional inner product space. Let $F^i$ for $i=0 ,\ldots, n$ be $C^{2+}$ be functions mapping $[0,T]\times {\cal V}$ to $\R$ and let $q^i:[0,T]\to \R$ be piecewise continuous.
Suppose that all the $F^i$ satisfy the equation
\begin{equation}
\frac{\partial F^i}{\partial t}(t,X_t)\, dt + \frac{1}{2}(\nabla \nabla F^i(t,X_t)) d[X]_t
{\color{red}-} A(t,X) \nabla F^i(t,X_t) dt = 0
\label{eq:blackScholesQuadratic}
\end{equation}
for some path-dependent ${\cal V}^*$-valued function $A$, where $\nabla$ denotes the covariant derivative  in the ${\cal V}$ direction.
Suppose that for all $t \in [0,T)$
\begin{equation}
\sum_{i=0}^n q^i_t (\nabla F^i)(t,X_t) = 0.
\label{eq:deltaHedgeQuadratic}
\end{equation}
Then
\begin{equation}
\lim_{N\to \infty} \sum_{[u,v] \in \pi_N} \sum_{i=0}^n  q^i_u( F^i(v, X_v)-F^i(u, X_u)) = 0.
\end{equation}
\end{theorem}
\begin{proof}
We prove the theorem for $[0,T)$ and for the terminal time $T$ we follow the same trick as \ref{cor:gammaHedging}.
\par Take $X^1=t$ and $X^2=X$ and apply the argument of Theorem \ref{thm:mainTheorem} with ${\cal V}^1=\R$ and ${\cal V}^2={\cal V}$. We obtain a similar Taylor expansion identity but without the absolute value signs:
{\color{red}
\begin{align}\label{eq:taylorExpansionNoAbs}
&
\sum_{[u,v]\in\pi_N} \sum_{i=0}^n q_u^i\bigl(F^i(X_v^1,X_v^2)-F^i(X_u^1,X_u^2)\bigr)
\nonumber \\ &= \sum_{[u,v]\in\pi_N} 
\sum_{\substack{(a_1,a_2)\in\mathcal I}}\sum_{i=0}^n\frac{1}{a_1!\,a_2!}q_u^i(\nabla^1_{X^1_v-X^1_u})^{\otimes a_1}(\nabla^2_{X^2_v-X^2_u})^{\otimes a_2}F^i(X_u^1,X_u^2)  \nonumber\\
& \qquad + \sum_{[u,v]\in\pi_N} \sum_{i=0}^nq_u^i  \int_0^1 (1-\lambda) \sum_{a_1+a_2=2} \frac{2!}{a_1!\,a_2!} \Big[ (\nabla^1_{X^1_v-X^1_u})^{\otimes a_1} (\nabla^2_{X^2_v-X^2_u})^{\otimes a_2} F^i\bigl(\gamma_{u,v}(\lambda)\bigr)
\nonumber
\\&\qquad\hspace{6cm}{\color{red}-(\nabla^1_{X^1_v-X^1_u})^{\otimes a_1} (\nabla^2_{X^2_v-X^2_u})^{\otimes a_2} F^i(X_u^1,X_u^2)\Big] d\lambda
},
\end{align}
where $\gamma_{u,v}(\lambda)$ is given by equation \eqref{eq:defLowerCaseGamma}}. Let us write $T^N_{a_1,a_2}$ for the term in equation 
$\eqref{eq:taylorExpansionNoAbs}$ corresponding to given values
of $a_1$ and $a_2$, that is:
\begin{align}
T^N_{a_1,a_2}:= \sum_{[u,v] \in \pi_N} \sum_{i=0}^n
 \frac{1}{a_1! \,a_2!} q^i_u
(\nabla^1_{X^1_v-X^1_u})^{\otimes a_1} (\nabla^2_{X^2_v-X^2_u})^{\otimes a_2}  F^i ({\color{red}X^1_u, X^2_u}). 
\label{eq:termsT}
\end{align}

Since the $\nabla$ of the statement is the gradient in the $\mathcal V$ variable, it coincides with
the $\mathcal V^2$-component $\nabla^2$ of the gradient on $\R\oplus\mathcal V$, while
$\nabla^1F^i=\partial_tF^i$. Hence, by definition of $T^N_{0,2}$ and equation
\eqref{eq:quadraticVariationEquivalence} (applied to the piecewise continuous function
$u\mapsto\tfrac12\sum_iq^i_u\nabla^{\otimes2}F^i(u,X_u)$),
\begin{align*}
T^N_{0,2}
&={\color{red}
\sum_{[u,v] \in \pi_N} \sum_{i=0}^n
 \frac{1}{2} q^i_u
(\nabla \nabla F^i)(X^1_u,X^2_u) (X^2_v-X^2_u)^{\otimes 2}} \\
&= 
\sum_{[u,v] \in \pi_N} \sum_{i=0}^n
 \frac{1}{2} q^i_u
(\nabla \nabla F^i)(X^1_u,X^2_u) ([X]_v-[X]_u) + {\color{red}o(1)}. 
\end{align*}
{\color{red}where $o(1)$ denotes a term which tends to zero as $N \to \infty$.}
{\color{red}Applying Lemma \ref{lemma:riemannSums}
to each $i$ and summing over $i=0 \ldots n$ of equation \eqref{eq:blackScholesQuadratic},
we find that  
\begin{align*}
\sum_{[u,v] \in \pi_N} \sum_{i=0}^n \Bigg[&
    q^i_u \frac{\partial F^i}{\partial t}(u,X_u) (v-u) + \frac{1}{2} q^i_u (\nabla \nabla F^i)(X^1_u,X^2_u) ([X]_v-[X]_u) \\
   &\qquad\qquad - q^i_u A(u,X) \nabla F^i(u,X_u) (v-u)
   \Bigg] = o(1).
\end{align*}
Hence
{\color{red}
\begin{align*}
T^N_{0,2}&=
\sum_{[u,v] \in \pi_N} \sum_{i=0}^n q^i_u  \left[
  A(u,X) \nabla F^i(u,X_u) (v-u) - \frac{\partial F^i}{\partial t}(u,X_u) (v-u) \right] + o(1) \\
&=
\sum_{[u,v] \in \pi_N} \sum_{i=0}^n q^i_u \left[
   A(u,X) \nabla F^i(u,X_u) (v-u) - \nabla^1_{X^1_v-X^1_u} F^i(u,X_u)\right] + o(1).
\end{align*}
}
We then have
\begin{align*}
T^N_{1,0}+T^N_{0,2}=
\sum_{[u,v] \in \pi_N} \sum_{i=0}^n
  q^i_u A(u,X) \nabla F^i(u,X_u) (v-u) + {\color{red}o(1)}.
\end{align*}
But this now vanishes \color{red}{as $N\to \infty$} by equation \eqref{eq:deltaHedgeQuadratic}.}
All other terms in equation \eqref{eq:taylorExpansionNoAbs} vanish {\color{red}in the limit} by our assumptions {\color{red}exactly as in the proof of Theorem \ref{thm:mainTheorem}} and the result follows.
\end{proof}

The financial interpretation of this Theorem is that classical delta-hedging
allows one to replicate European options in a diffusion model almost surely.

\section{Asian options}

In this section, we will illustrate how our results can be applied to replicate Asian options. For simplicity, we only consider the Black--Scholes case.

\begin{theorem}
Let $\pi$ be a partition sequence with mesh tending to 0. Suppose that $X \in C([0,T], \R)$
has vanishing $p^{th}$ variation along
$\pi$ for some $p<3$.
Let $F^i$ for $i=0 \ldots n$ be $C^{2+}$ be functions mapping $[0,T]\times \R \times \R$ to $\R$ and let $q^i:[0,T]\to \R$ be piecewise continuous.
Define
\begin{equation}
I_t = \int_0^t X_u \, du.
\label{eq:defI}
\end{equation}
Suppose that for some $\sigma>0$, all the $F^i$ satisfy the equations 
\begin{equation}
\frac{\partial F^i(t,I_t,X_t)}{\partial t} + X_t \frac{\partial F^i(t,I_t,X_t)}{\partial I} + \frac{1}{2} \sigma^2 X_t^2 \frac{\partial^2 F^i(t,I_t,X_t)}{\partial X^2}
= 0, \quad\forall i, \forall t\in [0,T]
\label{eq:blackScholesAsian}
\end{equation}
and,
\begin{equation}
\sum_{i=0}^n q^i_t \frac{\partial F^i}{\partial X}(t,I_t,X_t) = 0\quad   \forall t \in [0,T].
\label{eq:deltaHedgeAsian}
\end{equation}
If in addition either
\begin{enumerate}[(i)]
\item the portfolio is gamma-neutral. i.e., 
\begin{equation}
\sum_{i=0}^n q^i_t \frac{\partial^2 F^i}{\partial X^2}(t,I_t,X_t) = 0,\quad \forall t\in[0,T]
\label{eq:gammaHedgeAsian}
\end{equation}
or,
\item $X$ has finite  quadratic variation and the quadratic variation of $X$
satisfies
\begin{equation}
[X]_t = \int_0^t \sigma^2 X^2_u \, du
\end{equation}
\end{enumerate}
then for any $t\in [0,T]$;
\begin{equation}
\lim_{N\to \infty} \sum_{[u,v] \in \pi_N\cap[0,t]} \sum_{i=0}^n  q^i_u( F^i(v, {\color{red}I_v}, X_v)-F^i(u, {\color{red}I_u}, X_u)) = 0.
\end{equation}
\end{theorem}
\begin{proof}
To apply our argument to this case, we will take $X^1=(t,I)$ and $X^2=X$
and once again consider the terms
in equation \eqref{eq:taylorExpansionNoAbs}.
We define the terms $T^N_{a_1,a_2}$ using
equation \eqref{eq:termsT} as before.
\begin{align*}
T^N_{1,0} &=   
 \sum_{[u,v] \in \pi_N} 
\sum_{i=0}^n q^i_u
\nabla^1_{X^1_v-X^1_u} F^i  \\
&= 
\sum_{[u,v] \in \pi_N} 
\sum_{i=0}^n q^i_u \left[
\frac{\partial F^i(u,I_u,X_u)}{\partial t}
(v-u)
+ \frac{\partial F^i(u,I_u,X_u)}{\partial I}(I_v-I_u) 
\right] \\
&= \sum_{[u,v] \in \pi_N} 
\sum_{i=0}^n q^i_u \left[
\frac{\partial F^i(u,I_u,X_u)}{\partial t}
+ X_u \frac{\partial F^i(u,I_u,X_u)}{\partial I} 
\right] (v-u) + {\color{red}o(1)}
\end{align*}
by equation
\eqref{eq:defI} and Lemma \ref{lemma:riemannSums}.
Using equation \eqref{eq:blackScholesAsian} we find
that
\begin{align*}
T^N_{1,0} + T^N_{0,2}&= 
\sum_{[u,v] \in \pi_N} 
\sum_{i=0}^n q^i_u
\frac{1}{2}
\frac{\partial^2 F^i(u,I_u,X_u)}{\partial X^2}((X_v-X_u)^2 - \sigma^2 X_u^2 (v-u)) + {\color{red}o(1)}.
\end{align*}
{\color{red}
If condition (i) holds, one sees immediately that this term will vanish. If condition (ii) holds, we have by \eqref{eq:quadraticVariationEquivalence} that
\begin{align*}
T^N_{1,0} + T^N_{0,2} &= 
\sum_{[u,v] \in \pi_N} 
\sum_{i=0}^n q^i_u
\frac{1}{2}
\frac{\partial^2 F^i(u,I_u,X_u)}{\partial X^2}(\sigma^2 X_u^2 (v-u) - \sigma^2 X_u^2 (v-u)) + {\color{red}o(1)}.
\end{align*}
which vanishes in the limit $N\to \infty$.
}
All other terms in equation
\eqref{eq:taylorExpansion}
vanish in the limit by our assumptions.
\end{proof}

\section{Replicating path-dependent functionals}

\begin{definition}
Let $p \in {\mathbb R}_{\geq 1}$ and $({\cal V} , \|\cdot\|)$ be a real normed vector space. A continuous function $R:[0,T]\times[0,T] \to {\cal V}$ is said to have {\em vanishing $p$\textsuperscript{th}-variation} along a sequence of partitions $\pi = (\pi_N)_{N \geq 1}$ if
\[
\lim_{N \to \infty} \sum_{[u,v] \in \pi_N} \|R_{u,v}\|^p =0.
\]
\end{definition}

Given a path $X \in {\cal A}_{\pi}$ define a path $(X_{\wedge t,s})_{s \in [0,T]}$ by $X_{\wedge t, s}=X_{t \wedge s}$.
Given a subset $\bar{U} \subseteq{\cal A}_{\pi}$ write
\[
\bar{U}_{X_{\wedge t}}:=\{
X^\prime \in \bar U \mid X^\prime(s)=X(s) \quad \forall s\leq t
\}.
\]

\begin{definition}
If $\bar{U} \subseteq {\cal A}_{\pi}$ and ${\cal V}$ is a vector space, then a ${\cal V}$-valued path-dependent functional on $\bar{U}$ is a function
$\bar{F}:[0,T] \times \bar{U} \to  {\cal V}$ satisfying $\bar{F}(t,X)=\bar{F}(t,X_{\wedge t})$.
We write $\Lambda_{\bar{U}}({\cal V})$ for the space of pairs $(t, X)$, where $X$ is a ${\cal V}$-valued path
and $t$ is a time in $[0,T]$, modulo the equivalence relation $\sim$ where $(t,X)\sim(t^\prime,X^\prime)$ if and only if $t=t^\prime$ and $X_{\wedge t}=X^\prime_{\wedge t}$.  It is equipped with the metric $d_\infty$ given by
\[
d_\infty((t,X),(t^\prime,X^\prime)):=
|t-t^\prime| + \| X_{\wedge t} - X^\prime_{\wedge t^\prime}\|_\infty.\]
\end{definition}
\begin{definition}
Let $\bar{F}$ be a path-dependent functional on $\bar{U}$.
If $\Delta^\Theta \in $ and $\Gamma$ are also path-dependent functionals on $\bar U$, we say
$(\Delta^\Theta, \Gamma)$ {\em control} $\bar F$ along $\pi$ if for
all $X \in \bar U$
\begin{equation}
R_{u,v}^X:= \bar{F}(v,X)-\bar{F}(u,X) - \Delta^\Theta(u,X)\cdot(X_v-X_u) - \frac{1}{2}\Gamma(u,X) \cdot (X_v - X_u) \odot (X_v - X_u)
\label{eq:defR}
\end{equation}
has vanishing $1$-th variation along $\pi$. 

We say that ${\bar{F}}$ is a {\em pricing functional} on $(\bar{U},\pi)$ if it is controlled by some control $(\Delta^\Theta,\Gamma)$.
\end{definition}

Note that we work with remainders $R^X_{u,v}$ which are continuous. We study pricing functional within this continuous class, this will be sufficient to prove our results, but they could be extended by working with c\`adl\`ag paths.

The reader familiar with rough-path theory will see that this can be interpreted as requiring that $(\Delta^\Theta(t,X), \Gamma(t,X))$ supply a form of second-order Gubinelli derivative for $F$.
Our definition also takes inspiration from the functional It\^o calculus in that
we require the $\Delta$ and $\Gamma$ to be path-dependent functionals. The use of path-dependent functionals will allow us
to operationalize causality.

With this definition of a pricing-functional, the next result is trivial.

\begin{lemma}
Suppose that $\bar{F}^i$, $i=0,\ldots n$ are pricing functionals
for $(\bar{U},\pi)$, $X \in \bar U$ and $q^i_t:[0,T] \to {\mathbb R}$ are bounded functions satisfying for all $t\in [0,T)$,
\[
\sum_{i=0}^n q^i_t (\Delta_i^\Theta(t,X), \Gamma_i(t,X))=0,
\]
then
\[
\lim_{N\to \infty} \sum_{[u,v] \in \pi_N} \left| \sum_{i=0}^n q^i_u( \bar F^i(v,X)-\bar F^i(u,X)) \right| =0.
\]
\label{lemma:hedgingGreeks}
\end{lemma}

Given Lemma \ref{lemma:hedgingGreeks} above, Theorem \ref{thm:mainTheorem} is equivalent
to the theorem below. We prove the equivalence by deducing the Theorem \ref{thm:markovianAreFunctionals} from Theorem \ref{thm:mainTheorem} rather than by repeating the Taylor series argument used to prove Theorem \ref{thm:mainTheorem}.

\begin{theorem}
Let $U \subseteq {\cal V}^1 \oplus {\cal V}^2$ be an open subset
and let $F:U\to {\mathbb R}$ be a $C^{2+}$ function. Define $\bar{F}(t,X)=F(t,X_t)$. Define
\begin{equation}
\Delta^\Theta(F)(t,X) := (\nabla^1F(t,X), \nabla^2 F(t,X))
\label{eq:deltaMarkov}
\end{equation}
\begin{equation}
\Gamma(F)(t,X) := \nabla^2 \nabla^2 F(t,X)
\label{eq:gammaMarkov}
\end{equation}
Let $\bar{U} \subseteq{\cal A}_{\pi}$ be the set of $U$-valued
paths. Then $(\Delta^\Theta(F), \Gamma(F))$ control $\bar{F}$.
\label{thm:markovianAreFunctionals}
\end{theorem}
\begin{proof}
Define $F^{\alpha,i}_\Delta(t,X)=X^{\alpha,i}_t$ where $\alpha \in \{1,2\}$ and $X^{\alpha,i}_t$ are the components of $X^\alpha_t$ with respect to an orthonormal basis $\{e^\alpha_{i}\}_{i=1\ldots n_\alpha}$ for ${\cal V}^\alpha$. Define $F^{i,j}_\Gamma(t,X)=X^{2,i}_t X^{2,j}_t$.

Write
\[
{\mathcal F}:=\{F_\Delta^{\alpha,i}\}_{\alpha\in\{1,2\},i=1\ldots n_\alpha} \cup \{ F_\Gamma^{i,j}\}_{i=1\ldots n_2, j=i\ldots n_2}.
\]
It follows from the exact, finite Taylor expansions for the elements
$f \in {\mathcal F}$ that $(\Delta^\Theta(t,f),\Gamma(t,f))$ defined as in equations \eqref{eq:deltaMarkov} and \eqref{eq:gammaMarkov} control $f$. For all $t$, the set $\{(\Delta^\Theta(t,f), \Gamma(t,f))\}_{f \in {\cal F}}$ is a basis for ${\cal V}^1 \oplus {\cal V}^2 \oplus ({\cal V}^2 \odot {\cal V}^2)$. Hence we can uniquely choose bounded quantities $(q^{f,X}_t)_{f \in {\mathbb F}}$ at each time such that
\begin{equation}
(\Delta^{\Theta}(F)(t,X),
\Gamma(F)(t,X))
=\sum_{f \in {\mathcal F}} q^{f,X}_t (\Delta^{\Theta}(t,f), \Gamma(t,f)).
\label{eq:greeksEquality}
\end{equation}
Hence we may apply Theorem \ref{thm:mainTheorem}
using $f \in {\mathcal F}$ as the hedging functions and $F$ as the function to replicate to find
\begin{equation}
\lim_{N \to \infty}
\sum_{[u,v] \in \pi_N} \left|-F(v,X_v)+F(u,X_u)
+\sum_{f \in {\mathcal F}} q^{f,X}_u (f(v,X_v)-f(u,X_u))
\right| = 0.
\label{eq:mainTheoremConclusion}
\end{equation}
The term $\sum_{f \in {\mathbb F}} q^{f,X}_t (f(v,X_v)-f(u,X_u))$ is quadratic in $X_v-X_u$ and by equation \eqref{eq:greeksEquality} has greeks equal to
$(\Delta^\Theta(F)(t,X),\Gamma(F)(t,X))$.
Since the first and second derivatives of a quadratic determine the quadratic
\begin{equation}
\begin{split}
\sum_{f \in {\mathbb F}} q^{f,X}_u (f(v,X_v)&-f(u,X_u)) \\
&= 
\Delta^\Theta(F)(u,X)\cdot(X_v-X_u) + \frac{1}{2}\Gamma(F)(u,X) \cdot (X_v - X_u) \odot (X_v - X_u).
\end{split}
\label{eq:coefficientMatching}
\end{equation}
Combining equations
\eqref{eq:mainTheoremConclusion},
\eqref{eq:coefficientMatching}
and comparing with the definition
of $R_{u,v}$ in \eqref{eq:defR} we conclude
\[
\lim_{N \to \infty}
\sum_{[u,v] \in \pi_N} \left|-R^X_{u,v}
\right| = 0.
\]
\end{proof}

When Theorem \ref{thm:markovianAreFunctionals} is applied to European option prices in the Black--Scholes model with  ${\cal V}^1 = {\mathbb R}$ containing the time component of the signal and ${\cal V}^2 = {\mathbb R}$ containing the stock price, we see that $\Delta^\Theta=(\frac{\partial F}{\partial t}, \frac{\partial F}{\partial S})$, which explains why we call this term $\Delta^\Theta$ rather than simply $\Delta$.

\begin{definition}
A path-dependent functional $\bar{F} \in \lambda_{\bar{U}}{\cal V}$
is {\rm c\'adl\'ag} if for all $X \in \bar{U}$, $F(t,X)$ is a c\'adl\'ag function of $t$. We say that $F$ is {\em bounded on bounded sets} if
\[
\forall K>0 \quad \sup \{\sup_{t \in [0,T]} \|F(t,X)\| \mid X \in {\bar{U}}, \|X\|_\infty<K \} < \infty.
\]
We write ${\cal Q}_{\bar{U}}$ for the set of path-dependent functionals which are bounded on bounded sets.
\end{definition}

From this point forward, let us assume that our signal includes a time component. Let $\underline{t}:{\cal V}^1 \oplus {\cal V}^2 \to {\mathbb R}$ be the linear functional defining this time component; it should be zero on the ${\cal V}^2$ component. Write $\underline{t}^*$ for the vector satisfying $\langle {\underline{t}}^*, v\rangle=\underline{t}(v)$. Write ${\cal A}_{\pi,\underline{t}}$ for the subset of ${\cal A}_{\pi}$ consisting of paths $X$ with $\underline{t}(X_t)=t$.

\begin{definition}
A {\em pricing model} consists of a set of pricing functionals $\{\bar{F}^i\}_{i=1 \ldots n}$ defined on $\bar U \subseteq {\cal A}_{\pi, \underline{t}}$.

We say that $\bar{F}^0$ can be {\em replicated} in the pricing model
$\{\bar{F}^i\}_{i=1 \ldots n}$ if there exists $q^i\in {\cal Q}_{\bar{U}}$
such that the equation
\[
\bar F^0(s,X')-\bar F^0(t,X')
=\lim_{N\to\infty}\sum_{\substack{[u,v]\in\pi_N\\ t\le u,\ v\le s}}
\sum_{i=1}^n q^i(u,X')\big(\bar F^i(v,X')-\bar F^i(u,X')\big).
\]
holds for all $t<s \in [0,T]$ and $X' \in \bar{U}$.

The {\em span} of a pricing model, $\langle\bar{F}^i\rangle_{i=1\ldots n}$,  is the set of all pricing functionals that can be replicated on the common domain $\bar{U}$.
Given $X \in \bar{U}$, $t \in [0,T)$ we define
\[
\langle\bar{F}^i \mid X_{\wedge t}\rangle_{i=1\ldots n}
:=\langle \bar{F}^i_{|\bar{U}_{X_{\wedge t}}}\rangle_{i=1\ldots n}.
\]

A pricing model has a {\em strong arbitrage} at $(t,X) \in [0,T) \times \bar{U}$
if there exists 
\[
\bar{F}^{0} \in \langle\bar{F}^i \mid X_{\wedge t}\rangle_{i=1\ldots n}
\]
such that for all $X^\prime \in \bar{U}_{X_{\wedge t}}$ there exists
$\delta>0$ such that
$\bar{F}^0(t,X^\prime)=0$ and $\bar{F}^0(s,X^\prime)>0$ for $s \in (t,t+\delta]$.
\end{definition}

We call this a strong arbitrage because it guarantees a profit over an arbitrarily small time period once one knows that the state at time $t$ is given by $X_{\wedge t}$. Classical arbitrage only involves a non-zero probability of a profit and no probability of a loss.

\begin{theorem}
Suppose that $\{\bar{F}^i\}_{i=1 \ldots n}$ is a pricing model
and that $\bar{F}^i$ has controls which are continuous at $(t,X) \in [0,T) \times {\cal A}_{\pi,\underline{t}}$. The pricing model has a strong arbitrage at $(t,X)$ if the span of the controls at $(t,X)$ has full rank.
\label{theorem:arbitrage}
\end{theorem}
\begin{proof}
If the span of the controls has full rank, $r$, we can choose $r$ elements of the pricing model which have full rank. Choosing a basis for the space of controls, the determinant of the coordinate matrix of the controls of these $r$ elements is non-zero at at $(t,X)$ and so by the continuity of the controls, is bounded away from zero in a closed ball $B$ of $(t,X)$ in $\bar{U}_{X_{\wedge t}}$. We can moreover choose the ball so that the norms of the controls of the $r$ selected elements are bounded above. Hence we can define continuous functions $\bar{q}^i$ on $B$ which are zero except on the $r$ chosen elements and such that for all $(t^\prime, X^\prime) \in B$
\[
(\underline{t}^*,0)=\sum_{i=1}^n \bar{q}^i(t^\prime,X^\prime)(\Delta^\Theta(\bar{F}^i)(t^\prime,X^\prime),\Gamma(\bar{F}^i)(t^\prime,X^\prime))
\]
and moreover these functions will be bounded on $B$.

Define ${q}^i$ on $\bar{U}_{X_{\wedge t}}$ to equal $\bar{q}^i$ upto the first exit time of the ball $B$ and 0 outside. This ensures that $q^i(t,X)$ is c\'adl\'ag. As $q^i(t,X)$ is bounded, $q^i \in {\cal Q}_{\bar{U}}$.

For any $X^\prime \in B$, choose $\delta>0$ such that $t+\delta$ to is time $X^\prime$ exists the ball, or $\delta=T-t$ if $X^\prime$ never exits. By Lemma \ref{lemma:hedgingGreeks}, the
$q^i$ replicate the function $\bar{F}^0(t,X):=\underline{t}(X)-t$
on $\bar{U}_{X_{\wedge t}}$.
\end{proof}

Theorem \ref{theorem:arbitrage} implies that the controls within any pricing model without arbitrage on an open set $U$ must obey a linear relation.

In the example given in Corollary \ref{cor:gammaHedging} of hedging using the stock and a single option, each hedging
option $F^i$ has controls satisfying the Black--Scholes PDE,
equation \eqref{eq:blackScholesPDE}. 
There are two hedging options with linearly independent controls
and a three-dimensional space of controls in this example. So any 
additional pricing-functional that one might add to this model will introduce an arbitrage unless its controls also satisfy the Black--Scholes PDE.

We can characterise the example of Corollary \ref{cor:gammaHedging}
using the symmetry properties of its linear relation. This is the content of our next Lemma, which shows that if we assume i) a linear relation exists, ii) the stock and at least one call is traded, iii) the set of linear relations is invariant under time-translations and rescaling of the stock price, then the linear relation must be given by the Black--Scholes PDE.

\begin{lemma}
Given $\ell = (a,b,c) \in {\mathbb R}^3$ and $(t,s)\in {\mathbb R}\times {\mathbb R}_{>0}$, let
\[
{\cal L}^{\ell}_{t,s}F = a \frac{\partial F}{\partial t}(t,s) + b  \frac{\partial F}{\partial s}(t,s) + c \frac{\partial^2 F}{\partial s^2}(t,s)
\]
be a linear operator. Write $L_{t,s}$ for the space of such linear operators at $(t,s)$.
Let $T_h:L_{t,s} \to L_{t+h,s}$ by $T_h \ell = \ell$, so that $T_h$ represents the action of time-translation $t \mapsto t+h$ on linear operators.
Let $D_\lambda:L_{t,s} \to L_{t, \lambda s}$ by
$(a,b,c) \mapsto (a,\lambda b, \lambda^2 c)$,
so that $D_\lambda$ represents the action of 
the dilation $s \mapsto \lambda s$ on linear operators.
Suppose that $N_{t,s} \subseteq L_{t,s}$ is a choice of linear subspace at each $(t,s)$ which is invariant under time-translation and dilations and not equal to $\{0\}$.
Let ${\cal F}$ be the set of $C^3$ functions on $(0,T) \times {\mathbb R}_{>0}$ which are of at most polynomial growth in $s$ uniformly in $t$, and which satisfy
\[
{\cal L}^\ell_{t,s}F = 0
\]
for all $(t,s) \in (0,T) \times {\mathbb R}_{>0}$
and $\ell \in N_{t,s}$. If $F^1(t,s)=s \in {\cal F}$
and if there exists $F^2 \in {\cal F}$ with $\lim_{t \to T}F^2(t,s)=\max\{s-K,0\}$ 
then for every $(t,s)$, $N_{t,s}$ is spanned
by an operator of the form
\[
{\cal L}^{\ell}_{t,s}F =  \frac{\partial F}{\partial t}(t,s) +  c s^2 \frac{\partial^2 F}{\partial s^2}(t,s)
\]
for some constant $c>0$, independent of $(t,s)$.
\label{lemma:blackScholesBySymmetry}
\end{lemma}
\begin{proof}
Since $F^1(t,s) \in {\cal F}$, we compute that $b=0$ for all linear relations in $N_{t,s}$.

Suppose that $(0,0,1) \in N_{0,1}$ then $(0,0,s^2)\in N_{t,s}$ by our invariance hypothesis, so by linearity $(0,0,1) \in N_{t,s}$. This implies $\frac{\partial^2 F}{\partial s^2} = 0$ for all $F \in {\cal F}$. Hence all such $F$ are affine in $s$ and this contradicts the assumption that $F^2 \in {\cal F}$.

Suppose that $(1,0,0) \in N_{0,1}$ then $\frac{\partial F}{\partial t} =0$ for all $F \in {\cal F}$. This implies $F^2$ is equal to its payoff, which is not differentiable. Hence $F^2 \notin C^3$, a contradiction.

We know that $N_{0,1}$ is spanned by elements
of the form $(a,0,c)$. If it were two-dimensional it would contain $(0,0,1)$ and $(1,0,0)$. Hence it is one-dimensional and so is spanned by an element $(1,0,c)$
for $c \neq 0$. By the invariance hypothesis, $N_{t,s}$ is spanned by
$D_s T_t (1,0,c) = (1,0,cs^2)$ for every $(t,s)$. Hence ${\cal F}$ consists of solutions either to the Black--Scholes PDE ($c > 0$) or the time-reversed Black--Scholes PDE ($c < 0$).
Any element of ${\cal F}$ solving the time-reversed Black--Scholes equation subject to the polynomial growth bound extends smoothly to time $T$, contradicting the fact that the limit of $F^2$
is equal to a call option payoff and so non-smooth.
\end{proof}

The same argument used to prove Theorem \ref{theorem:arbitrage} can be applied to models where
$\bar U\subseteq\mathcal A_{\pi,\underline t}$ consists only of paths with known quadratic variation,
\[
[X]_t=\int_0^t \mathrm{Var}(s,X_{\wedge s})\,ds ,
\]
for some $(\mathcal V^2\odot\mathcal V^2)^*$-valued path-dependent functional $\mathrm{Var}$ such that
$s\mapsto\mathrm{Var}(s,X_{\wedge s})$ is piecewise continuous for every $X\in\bar U$. 

Let $\bar F^i$, $i=0,\ldots,n$, be pricing functionals on $(\bar U,\pi)$ with controls
$(\Delta^\Theta(\bar F^i),\Gamma(\bar F^i))$ such that $s\mapsto\Gamma(\bar F^i)(s,X)$ is piecewise
continuous, and let $q^i:[0,T]\to\R$ be piecewise continuous. Define the {\em projected greeks}
of $\bar F^i$ to be $(\Pi(\bar F^i),0)$, where
\[
\Pi(\bar F^i)(t,X):=\Delta^\Theta(\bar F^i)(t,X)
+\frac12\big(\mathrm{Var}(t,X_{\wedge t})\,\Gamma(\bar F^i)(t,X)\big)\,t^*
\;\in\;(\mathcal V^1\oplus\mathcal V^2)^* .
\]
Since $\underline t(X_v-X_u)=v-u$ on $\mathcal A_{\pi,\underline t}$, we have from \eqref{eq:defR}
that for $X\in\bar U$,
\begin{align*}
\sum_{[u,v]\in\pi_N}\sum_{i=0}^n &q^i_u\big(\bar F^i(v,X)-\bar F^i(u,X)\big)\\
&=\sum_{[u,v]\in\pi_N}\Big(\sum_{i=0}^n q^i_u\,\Pi(\bar F^i)(u,X)\Big)(X_v-X_u)
+\sum_{[u,v]\in\pi_N}\sum_{i=0}^n q^i_u R^{i}_{u,v}\\
&\qquad+\frac12\sum_{[u,v]\in\pi_N}\Big(\sum_{i=0}^n q^i_u\,\Gamma(\bar F^i)(u,X)\Big)
\Big((X^2_v-X^2_u)^{\otimes 2}-\mathrm{Var}(u,X_{\wedge u})\,(v-u)\Big).
\end{align*}
The second sum on the right is bounded by
$\max_i\|q^i\|_\infty\sum_i\sum_{[u,v]\in\pi_N}|R^i_{u,v}|$, which tends to zero as $N\to\infty$
because $(\Delta^\Theta(\bar F^i),\Gamma(\bar F^i))$ control $\bar F^i$. For the third sum, the
function $u\mapsto\frac12\sum_iq^i_u\Gamma(\bar F^i)(u,X)$ is piecewise continuous, so by
\eqref{eq:quadraticVariationEquivalence} and Lemma~\ref{lemma:riemannSums} we calculate that
\begin{align*}
\sum_{[u,v]\in\pi_N}\sum_{i=0}^n
 &\frac{1}{2} q^i_u \Gamma(\bar{F}^i)(u,X)\,(X^2_v-X^2_u)^{\otimes 2} \\
 &=\sum_{[u,v]\in\pi_N}\sum_{i=0}^n
 \frac{1}{2} q^i_u \Gamma(\bar{F}^i)(u,X)\,([X]_v-[X]_u) + o(1) \\
&=\sum_{[u,v]\in\pi_N}\sum_{i=0}^n
 \frac{1}{2} q^i_u\, \mathrm{Var}(u, X_{\wedge u})\,\Gamma(\bar{F}^i)(u,X)\,(v-u) + o(1),
\end{align*}
where $o(1)$ denotes a quantity tending to zero as $N\to\infty$. Comparing the first and last
lines, the third sum is $o(1)$. Hence, if
\[
\sum_{i=0}^n q^i_t\,\Pi(\bar F^i)(t,X)=0\qquad\text{for all }t\in[0,T),
\]
then
\[
\lim_{N\to\infty}\sum_{[u,v]\in\pi_N}\sum_{i=0}^n q^i_u\big(\bar F^i(v,X)-\bar F^i(u,X)\big)=0 .
\]

Repeating the argument of 
Theorem \ref{theorem:arbitrage} 
shows that if the projected greeks
have maximal rank at $(t,X)$ then we obtain a strong arbitrage at $(t,X)$ for a pricing model on $\bar{U}$.

For example, take $X=(t,S)$
and $\mathrm{Var}(t,S)=\sigma^2 S^2_t$ for some constant $\sigma$, so $\bar{U}$ consists of paths
have quadratic variation corresponding to the Black--Scholes model.
Let ${\bar{F}^i}_{i=1\ldots n}$ be an arbitrage-free pricing model on $\bar{U}$ with $\bar{F}^1(t,S)=S_t$,
i.e.\ we are assuming that $S$ is a traded asset. In this case $\Delta^\Theta(\bar{F}^1)=(0,1)$
and $\Gamma(\bar{F}^1)=0$. So
the equation
\begin{equation}
\underline{t}( \Delta^\Theta(\bar{F}^1)(t,X) + \frac{1}{2} \sigma^2 S_t^2 \Gamma(\bar{F}^1)(t,X) \underline{t}^*)
= 0
\label{eq:blackScholesDisguised1}
\end{equation}
holds for the projected greeks.
If the pricing model is arbitrage free, all the projected greeks must satisfy some linear relation. Since the dimension of ${\cal V}^1 \oplus {\cal V}^2$ is $2$, the only linear relation satisfied by the projected greeks of $\bar{F}^1$ up to scale is \eqref{eq:blackScholesDisguised1}.
Hence
\begin{equation}
\underline{t}( \Delta^\Theta(\bar{F}^i)(t,X) + \frac{1}{2} \sigma^2 S_t^2 \Gamma(\bar{F}^i)(t,X) \underline{t}^*)
= 0
\label{eq:blackScholesDisguised}
\end{equation}
must hold for all $\bar{F}^i$ in the pricing model. Specialising to the Markovian case where $\bar{F}(t,S)=F(t,S_t)$, equation \eqref{eq:blackScholesDisguised}
is simply the Black--Scholes PDE
\eqref{eq:blackScholesPDE}.

\medskip

In general, if one interprets a linear relation in the controls as a path-dependent PDE, Theorem \ref{theorem:arbitrage} shows that, subject to the conditions of the theorem, some such PDE must hold in any pricing model that is free of arbitrage.

The Feynman--Kac theorem gives a correspondence between the solution of linear PDEs and the expectation of stochastic processes. Thus, for reasonable arbitrage-free pricing models, there will be an associated stochastic process. Thus, a classical risk-neutral pricing model emerges from our notion of a strong-arbitrage-free pricing model. 

We have restricted the quantities to lie in ${\cal Q}_{\bar{U}}$ in our definition of replication
to eliminate pathological strategies. In classical stochastic calculus, any bounded c\'adl\'ag process which is adapted with respect to a continuous semi-martingale is integrable with respect to the semi-martingale which generated the filtration. Given a strong arbitrage on $\bar{U}$, we can stop the strategy once $(X_s)_{s \in [t,T]}$ leaves a small ball around $X_t$ to obtain bounded $q^i$ which will still be c\'adl\'ag. Thus, if a sure-arbitrage occurs almost surely, this will give rise to a classical arbitrage in any semi-martingale model with paths in $\bar{U}$. 
This establishes the relationship between
our results and classical probabilistic arbitrage theory.

\bibliographystyle{apalike} 
\bibliography{bibliography}

\end{document}